\documentclass[12pt,a4paper]{article}

\usepackage[T2A]{fontenc}

\usepackage[intlimits]{amsmath}
\usepackage{amsfonts,amssymb,amscd,amsthm,cite}
\usepackage{amstext}

\usepackage{geometry}
\newtheorem{theorem}{Theorem}[section]

\newtheorem{lemma}[theorem]{Lemma}

\theoremstyle{definition}
\newtheorem{definition}[theorem]{Definition}

\theoremstyle{remark}

\newtheorem*{remark*}{Remark}

\numberwithin{equation}{section}
\numberwithin{figure}{section}

\begin{document}

\title{An operator algebra  associated with a pair of intersecting manifolds}
\author{D.A.~Loshchenova, A.Yu.~Savin, B.Yu.~Sternin}
\date{}

\maketitle

\begin{abstract}
Given a pair of smooth  transversally intersecting manifolds in some ambient manifold, we construct an operator algebra
generated by pseudodifferential operators and the (co)boundary operators associated with
the submanifolds. We show that this algebra has 18 types of additive generators.
Then we define the symbols of the operators in this algebra and obtain the composition formula.
\end{abstract}


\section{Introduction}

Let $X_0$ be a closed smooth manifold  and $X_1$, $X_2$ be its submanifolds
of arbitrary dimensions with a nonempty intersection.
We consider a natural class of boundary value problems associated with this geometry.
Namely, we consider operators  with boundary conditions posed on   $X_1$ and $X_2$.
These problems were studied, e.g., in the works~\cite{Ster2,Ster5,ZeSt1}
(see also~\cite{SaSt32,SaSt35}), where the Fredholm property
for some problems of such type was obtained, index formulas were proved (these formulas
involve contributions of   $X_1$ and $X_2$ considered as strata of the manifold with singularities 
$X_1 \cup X_2$). Later these results and methods were applied to study nonlocal problems with boundary conditions
on a smooth submanifold (see~\cite{Linh1,Losh1}).

In the present paper, we study several algebraic aspects of this theory.
More precisely, we consider an operator algebra multiplicatively generated by  pseudodifferential operators ($\psi$DOs) on the ambient manifold and on the submanifolds, and the (co)boundary operators associated with the submanifolds
(by a boundary operator we mean the restriction operator to a submanifold,
and by a coboundary operator we mean its dual, that is, an operator which extends  functions
on a submanifold to  distributions on the ambient manifold). We show that this algebra has
18 types of additive generators, and  general elements in this algebra can be written  
as the following $3\times 3$ matrices  
\begin{equation}\label{eq:111}
\mathcal{D} = \left(\rule{0ex}{8ex}
    \begin{array}{ccc}
        D_0 + G_1 + G_2 + M_0 & C_1 + C_1^\prime & C_2 + C_2^\prime \vspace{2ex} \\
        B_1 + B_1^\prime & D_1 + M_1 & T_{12} \vspace{2ex} \\
        B_2 + B_2^\prime & T_{21} & D_2 + M_2 \\
    \end{array}
\right)
\colon \mathcal{H}  \longrightarrow \mathcal{H}',
\end{equation}
where $\mathcal{H},\mathcal{H}'$ stand for direct sums of Sobolev spaces on   $X_0,X_1,X_2$
(of some orders for each of the manifolds).
The entries in~\eqref{eq:111} are of the following types:
\begin{itemize}
\item $D_0,D_1,D_2$ are $\psi$DOs on $X_0$, $X_1$, $X_2$, respectively;
\item $B_1,B_2$ and $C_1,C_2$ are boundary and coboundary operators (see~\cite{Ster1}),
    localized at $X_1$, $X_2$;
\item $G_1,G_2$ are Green operators (see, e.g.,~\cite{StSh53-1,NaSt10,NaSt13}),
    localized at $X_1,X_2$;
\item $M_0,M_1,M_2$ are Mellin operators (see, e.g.,~\cite{SaSt35}),
    localized at $X_1 \cap X_2$;
\item $T_{12},T_{21}$ are translators (see~\cite{Ster5}),
    localized at $X_1 \cap X_2$;
\item $B'_1,B'_2$ and $C'_1,C'_2$ are boundary and coboundary operators,
    localized at the intersection $X_1 \cap X_2$.
\end{itemize}
We note that while these operators were considered in the literature individually,
our approach allows one to study them from a unified point of view,
and this considerably simplifies the theory.
Moreover, our classification of operators~\eqref{eq:111}
is carried out in terms of the strata 
on which these operators are localized at.
We also define symbols of these operators and establish the composition formula.

The work is partially supported by RFBR grants
\textnumero 16-31-00176 and 16-01-00373.

\section{Statement of the problem}

Let $X_0 = \mathbb{R}^n$ with the standard coordinates $x_1,\dotsc,x_n$.
We consider two coordinate subspaces $X_1, X_2 \subset X_0$
such that $\dim X_k < \dim X_0$, $k = 1,2$, and $\dim (X_1 \cap X_2) >0$.
We will take appropriate components of the coordinates $x_1,\dotsc,x_n$ on the ambient space
as coordinates on $X_k$, $k = 1,2$.

For $k=1,2$ denote $n_k = \operatorname{codim}_{X_k} (X_1 \cap X_2)$, $\nu_k = \operatorname{codim}_{X_0} X_k$.
Also denote $\nu_3 = \operatorname{codim}_{X_0} (X_1 \cap X_2)$.

We associate the following operators with the triple $(X_0, X_1, X_2)$:
\begin{enumerate}
\item Pseudodifferential operators ($\psi$DOs) on $X_0, X_1, X_2$:
\begin{equation}\label{eq:pdo}
A_k\colon H^s(X_k) \longrightarrow H^{s-m}(X_k), \quad k = 0,1,2.
\end{equation}
Hereinafter  we consider only $\psi$DOs~\eqref{eq:pdo} with compactly supported Schwartz kernels.

\item Elementary boundary operators corresponding to $X_k$, $k=1,2$:
\begin{equation}\label{eq:b}
i^k\colon H^s(X_0) \longrightarrow H^{s-\nu_k/2}(X_k),
\quad u(y,z) \longmapsto u(y, 0),
\quad s-\nu_k/2 > 0,
\end{equation}
where $(y,z)$ are the coordinates on $X_0$ such that $X_k = \{y = 0\}$.

\item Elementary coboundary operators corresponding to $X_k$, $k=1,2$:
\begin{equation}\label{eq:cob}
i_k\colon H^{-s+\nu_k/2}(X_k) \longrightarrow H^{-s}(X_0),
\quad u(z) \longmapsto u(z)\otimes \delta(y),
\quad s-\nu_k/2 > 0,
\end{equation}
where
$\delta(y)$ stands for the Dirac delta function, and the coordinates $(y,z)$ are chosen as above.
\end{enumerate}

Below we always assume that all operators act between Sobolev spaces of some fixed orders
(we will denote these spaces by $H(Z)$, where $Z \subset X_0$ is a manifold,
omitting the order from the notation). Moreover, for compositions $D_1 D_2$
we assume that the domain of $D_1$ is equal  to the range of $D_2$,
that is, the corresponding Sobolev spaces are the same.

Now, for any $k, l \in \{0,1,2\}$, consider the linear space $\operatorname{Mor}_{k,l}$
which consists of operators $H(X_l)\to H(X_k)$ of some fixed order
and is multiplicatively generated by the operators~\eqref{eq:pdo},
\eqref{eq:b}, and~\eqref{eq:cob}.
More precisely, an element of $\operatorname{Mor}_{k,l}$ is a finite sum of operators of the form
\begin{equation}\label{defeq:simpleMorElement}
\mathcal{D}_{kl} = D_{k,i_1} D_{i_1,i_2}\ldots D_{i_N,l} \colon H(X_l) \longrightarrow H(X_k),
\end{equation}
where $D_{\alpha,\beta}\colon H(X_\beta) \rightarrow H(X_\alpha)$
is a composition of $\psi$DOs
and the elementary boundary operator (when $X_\alpha \subset X_\beta$)
or the elementary coboundary operator (when $X_\beta \subset X_\alpha$).
Now denote by $\operatorname{Mor}$ the direct sum
\begin{equation}\notag 
\operatorname{Mor} = \bigoplus_{k,l = 0,1,2} \operatorname{Mor}_{k,l}.
\end{equation}
Elements of this space are called \textit{morphisms} (cf.~\cite{Ster2}).

Further, for simplicity we consider only morphisms of order zero in the spaces
\begin{equation}\label{defeq:cH:oplus}
\mathcal{H} = \bigoplus_{k = 0}^{2} H^{s_k}(X_k)
\end{equation}
for some fixed $s_k \in \mathbb{R}$.
By construction such morphisms form an algebra with respect to the operator composition
(we denote this algebra also by $\operatorname{Mor}$).

It is an interesting problem to study operators (morphisms) from the algebra $\operatorname{Mor}$.
More precisely, it is necessary to examine the structure of these operators,
define their symbols and the notion of ellipticity, establish the Fredholm property
of elliptic operators (the finiteness theorem) and obtain the index formula.

In this paper, we carry out a classification of morphisms, define their symbols
and determine the composition formula (the symbol homomorphism).
Fredholm property  and index theorem will be studied elsewhere.

\section{Classification of morphisms}

The union $X_1 \cup X_2$ is a stratified  manifold with singularities in $X_0$.
Let us classify the elements of $\operatorname{Mor}$ by means of the strata they are localized at.
First, we introduce the corresponding concept.

Note that the space $\mathcal{H}$ (see~\eqref{defeq:cH:oplus}) is a $C^\infty_c(X_0)$-module
with respect to the multiplication by functions in $C^\infty_c(X_0)$ and their
restrictions to $X_1$ and $X_2$.
\begin{definition}
A morphism $\mathcal{D}\colon \mathcal{H} \rightarrow \mathcal{H}$ is \textit{localized} at a submanifold $Z \subset X_0$
if compositions $\varphi \mathcal{D}$ or $\mathcal{D} \varphi$ are operators of order $\leq \operatorname{ord}\mathcal{D}$
for any $\varphi \in C^\infty_c(X_0 \setminus Z)$.
\end{definition}

One can easily show that the elementary boundary and coboundary operators
are localized at the submanifolds they are associated with. This implies the following
\begin{lemma}\label{lemma:simpleMorElementLocalized}
The composition~\eqref{defeq:simpleMorElement} is localized at the intersection
\begin{equation}\label{eq:simpleMorElementLocalizedCap}
X_{k} \cap X_{i_1} \cap \dotsb \cap X_{i_N} \cap X_{l}.
\end{equation}
\end{lemma}

Lemma~\ref{lemma:simpleMorElementLocalized} allows us to classify
compositions~\eqref{defeq:simpleMorElement} into four classes,
according to the intersection~\eqref{eq:simpleMorElementLocalizedCap}:
this intersection is equal to one of the submanifolds $X_0, X_1, X_2$, and $X_1 \cap X_2$.
Let us examine the form of the operators from each of the four classes.

From now on, we represent the morphisms from $\operatorname{Mor}$ as $3\times 3$ matrix operators
acting in $\mathcal{H} \rightarrow \mathcal{H}$.

1. \textit{The intersection~\eqref{eq:simpleMorElementLocalizedCap} is equal to $X_0$.}
Clearly, this is the case if and only if
$$
X_{k} = X_{i_1} = \dotsb = X_{i_N} = X_{l} = X_0.
$$
Thus, we have a $\psi$DO on $X_0$. The corresponding matrix operator is
\begin{equation}\notag
\left(
    \begin{array}{ccc}
        D & 0 & 0 \\
        0 & 0 & 0 \\
        0 & 0 & 0 \\
    \end{array}
\right) \colon \mathcal{H} \longrightarrow \mathcal{H}.
\end{equation}

2. \textit{The intersection~\eqref{eq:simpleMorElementLocalizedCap} is equal to $X_1$.}
This is the case if and only if all the submanifolds in~\eqref{eq:simpleMorElementLocalizedCap}
are equal to either $X_0$ or $X_1$.
Thus, we deal with a composition of $\psi$DOs on $X_0$ and $X_1$ with (at least  one)
the elementary boundary or coboundary operators $i_1$, $i^1$.
The corresponding matrix operator is
\begin{equation}\notag
\left(
    \begin{array}{ccc}
        G_1 & C_1 & 0 \\
        B_1 & D_1 & 0 \\
        0 & 0 & 0 \\
    \end{array}
\right) \colon \mathcal{H} \longrightarrow \mathcal{H}.
\end{equation}
Here $D_1$ is always a $\psi$DO~on $X_1$ because the composition $i^1 A i_1$,
where $A$ is a $\psi$DO on $X_0$, is a $\psi$DO on $X_1$
(see~\cite{Ster1, NoSt1}).
Next, $B_1$ is  equal to
\begin{equation}\label{eq:pdoB}
    A_1 \, i^1 \, A_0,
\end{equation}
where $A_1$, $A_0$ are $\psi$DOs on $X_1$, $X_0$, respectively. Note that   ``long'' compositions
$$
    A_1 \, i^1 \,  A_0 \, i_1 \,  A_1^\prime \, i^1 \, A_0^\prime,
$$
where $A_1^\prime$, $A_0^\prime$ are $\psi$DOs on $X_1$, $X_0$, respectively,
can always   be represented as in~\eqref{eq:pdoB}, since, as noted above, $i^1 A_0 i_1$ is a $\psi$DO on $X_1$.
The operator $B_1$ is a \textit{boundary operator} localized at $X_1$. In a dual manner, $C_1$ is  
\begin{equation}\notag
    A_0 \, i_1 \, A_1.
\end{equation}
It is a \textit{coboundary operator} localized at $X_1$.
Finally, we have
$$
    G_1 = A_0 \, i_1 \, A_1 \, i^1 \, A_0^\prime,
$$
where $A_0,A_0^\prime$ are $\psi$DOs on $X_0$, and $A_1$ is a $\psi$DO on $X_1$.
This operator is called \textit{Green operator} localized at $X_1$.

3. \textit{The intersection~\eqref{eq:simpleMorElementLocalizedCap} is equal to $X_2$.}
This case is analogous to the previous one. One obtains matrix operators  
\begin{equation}\notag
\left(
    \begin{array}{ccc}
        G_2 & 0 & C_2 \\
        0 & 0 & 0 \\
        B_2 & 0 & D_2 \\
    \end{array}
\right) \colon \mathcal{H} \longrightarrow \mathcal{H},
\end{equation}
where $D_2$ is a $\psi$DO on $X_2$;
$B_2$ and $C_2$ are boundary and coboundary operators localized at $X_2$, respectively;
$G_2$ is a Green operator localized at $X_2$.

4. \textit{The intersection~\eqref{eq:simpleMorElementLocalizedCap} is equal to $X_1 \cap X_2$.}
In this case, one obtains matrix operators  
\begin{equation}\notag
\left(
    \begin{array}{ccc}
        M_0 & C_1^\prime & C_2^\prime \\
        B_1^\prime & M_1 & T_{12} \\
        B_2^\prime & T_{21} & M_2 \\
    \end{array}
\right) \colon \mathcal{H} \longrightarrow \mathcal{H}.
\end{equation}
Each term in this matrix is a composition with at least one (co)boundary  operator for $X_1$
and at least one (co)boundary   operator for $X_2$.
The operators $M_0$, $M_1$, $M_2$ are called \textit{Mellin operators} localized at $X_1 \cap X_2$;
the operators $B_1^\prime, B_2^\prime$ and $C_1^\prime, C_2^\prime$ are called
 boundary and coboundary operators localized at $X_1 \cap X_2$;
the operators $T_{12}$, $T_{21}$ are called  \textit{translators}
between $X_1$ and $X_2$ (see~\cite{Ster5}).

The operators of the above four classes constitute the set of
additive generators of the algebra $\operatorname{Mor}$.
Therefore, a general morphism is of the form
\begin{equation}\label{eq:generalMorphism}
\mathcal{D} = \left(\rule{0ex}{8ex}
    \begin{array}{ccc}
        D_0 + G_1 + G_2 + M_0 & C_1 + C_1^\prime & C_2 + C_2^\prime \vspace{2ex} \\
        B_1 + B_1^\prime & D_1 + M_1 & T_{12} \vspace{2ex} \\
        B_2 + B_2^\prime & T_{21} & D_2 + M_2 \\
    \end{array}
\right)
\colon \mathcal{H} \longrightarrow \mathcal{H}.
\end{equation}
It follows that the algebra $\operatorname{Mor}$ has $18$ types of additive generators.

\section{Symbols of morphisms}
\subsection{Symbols of $\psi$DOs and (co)boundary operators}
In this section, we define symbols of general morphisms in $\operatorname{Mor}$.
First, we define symbols of the generators on various submanifolds.

1. \textit{Symbols for $\psi$DOs.}
Let $A$ be a $\psi$DO on $X_0$ and $Z$ be a stratum in $X_0$.
Denote by $(y,z)$ the coordinates on $X_0$ such that $Z = \{y = 0\}$.
Denote by $(\eta,\zeta)$ the corresponding coordinates in the fibers of the bundle $T^*X_0$.
The principal symbol of $A$ is denoted by $A(y,z,\eta,\zeta)$.
Now, the \textit{symbol} $\sigma_Z(A)$ of $A$ on $Z$ is defined as the following operator-function
\begin{equation}\label{defeq:AZSmb}
\sigma_Z(A)(z,\zeta) = A\biggl(0,z,-i\frac{\partial }{\partial y},\zeta\biggr)\colon
H(\mathbb{R}^k_y) \longrightarrow H(\mathbb{R}^k_y),
\quad (z,\zeta) \in T^*_0Z.
\end{equation}
Note that~\eqref{defeq:AZSmb} is obtained by freezing the coefficients
of $A$ at a point on $Z$ and applying the Fourier transform with respect to the tangent variables.

The symbols of $\psi$DOs on $X_1$, $X_2$, $X_1 \cap X_2$ are defined similarly.

2. \textit{Symbols of elementary boundary operators.}
The \textit{symbol} $\sigma_{X_1}(i^1)$ of the operator $i^1$
on the stratum $X_1$ is the operator-function
\begin{equation}\notag
\sigma_{X_1}(i^1)(z,\zeta)\colon H(\mathbb{R}^{\nu_1}_y) \longrightarrow \mathbb{C},
\quad u(y) \longmapsto u(0),
\quad  (z,\zeta) \in T^*_0X_1,
\end{equation}
where $(z,y)$ are the coordinates on $X_0$ such that $X_1 = \{y = 0\}$.
To define the symbol of $i^1$ on the stratum $X_1 \cap X_2$,
we choose the local coordinates $(x,y,z)$ on $X_0$ such that
$X_1 = \{y = 0\}$, and $X_1 \cap X_2 = \{(x,y) = (0,0)\}$.
Then the \textit{symbol} $\sigma_{X_1 \cap X_2}(i^1)$ is the operator-function
\begin{equation}\notag
\sigma_{X_1 \cap X_2}(i^1)(z,\zeta)\colon H(\mathbb{R}_{x,y}^{\nu_3}) \longrightarrow H(\mathbb{R}_{x}^{n_1}),
\quad u(x,y) \longmapsto u(x,0),
\end{equation}
where $(z,\zeta) \in T^*_0(X_1 \cap X_2)$.

The symbols $\sigma_{X_2}(i^2)$ and $\sigma_{X_1 \cap X_2}(i^2)$ are defined along the same lines.

3. \textit{Symbols of coboundary operators.}
The symbol of the operator $i_1$ is defined in a dual manner to that of $i^1$.
More precisely, we set
\begin{equation}\notag
\sigma_{X_1}(i_1)(z,\zeta)\colon \mathbb{C}  \longrightarrow H(\mathbb{R}^{\nu_1}_y),
\quad q \longmapsto q\,\delta(y),
\quad  (z,\zeta) \in T^*_0X_1,
\end{equation}
and
\begin{equation}\notag
\sigma_{X_1 \cap X_2}(i^1)(z,\zeta)\colon H(\mathbb{R}_{x}^{n_1}) \longrightarrow H(\mathbb{R}_{x,y}^{\nu_3}),
\quad u(x) \longmapsto u(x) \otimes \delta(y),
\end{equation}
where $(z,\zeta) \in T^*_0(X_1 \cap X_2)$.

The symbols $\sigma_{X_2}(i^2)$ and $\sigma_{X_1 \cap X_2}(i^2)$ of $i^2$ are defined similarly.

\subsection{Symbols of general morphisms}

Let $Z$ be any of the strata $X_0$, $X_1$, $X_2$, $X_1 \cap X_2$.
\begin{definition}\label{def:morphismSmb}
The composition
\begin{equation}\label{defeq:simpleMorElementSmb}
\sigma_Z(\mathcal{D}_{kl}) = \sigma_Z(D_{k,i_1}) \, \sigma_Z(D_{i_1,i_2}) \, \ldots \, \sigma_Z(D_{i_N,l}),
\end{equation}
is called the \textit{symbol} $\sigma_Z(\mathcal{D}_{kl})$ of the morphism~\eqref{defeq:simpleMorElement} on $Z$.
Note that all the terms $\sigma_Z$ on the right hand side were defined above.
\end{definition}

Therefore, the following symbols are defined for a general morphism~\eqref{eq:generalMorphism}.
\begin{enumerate}
\item The symbol on $X_0$ is equal to the  symbol of the $\psi$DO component $D_0$:
\begin{equation}\notag
\sigma_{X_0}(\mathcal{D})(z,\zeta) = \sigma(D_0)(z,\zeta)\colon \mathbb{C} \longrightarrow \mathbb{C},
\quad (z,\zeta) \in T^*_0X_0.
\end{equation}

\item The symbol on $X_1$ is the operator-function
\begin{equation}\notag
\sigma_{X_1}(\mathcal{D})(z,\zeta) =
\left(\rule{0ex}{4ex}
    \begin{array}{cc}
        \sigma_{X_1}(D_0 + G_1) & \sigma_{X_1}(C_1) \vspace{1ex} \\
        \sigma_{X_1}(B_1) & \sigma(D_1) \vspace{1ex} \\
    \end{array}
\right)(z,\zeta), \quad (z,\zeta) \in T^*_0X_1 
\end{equation}
ranging in operators acting in the spaces
\begin{equation}\notag
\sigma_{X_1}(\mathcal{D})(z,\zeta) \colon
\begin{array}{c}
    H(\mathbb{R}^{\nu_1}) \\
    \oplus \\
    \mathbb{C}
\end{array}
\longrightarrow
\begin{array}{c}
    H(\mathbb{R}^{\nu_1}) \\
    \oplus \\
    \mathbb{C}
\end{array}.
\end{equation}

\item The symbol on $X_2$ is the operator-function
\begin{equation}\notag
\sigma_{X_2}(\mathcal{D})(z,\zeta) =
\left(\rule{0ex}{4ex}
    \begin{array}{cc}
        \sigma_{X_2}(D_0 + G_2) & \sigma_{X_2}(C_2) \vspace{1ex}\\
        \sigma_{X_2}(B_2) & \sigma(D_2) \vspace{1ex}\\
    \end{array}
\right)(z,\zeta),
\quad (z,\zeta) \in T^*_0X_2,
\end{equation}
with values in operators acting in the spaces
\begin{equation}\notag
\sigma_{X_2}(\mathcal{D})(z,\zeta) \colon
\begin{array}{c}
    H(\mathbb{R}^{\nu_2}) \\
    \oplus \\
    \mathbb{C}
\end{array}
\longrightarrow
\begin{array}{c}
    H(\mathbb{R}^{\nu_2}) \\
    \oplus \\
    \mathbb{C}
\end{array}.
\end{equation}

\item Finally, the symbol on the intersection $Z = X_1 \cap X_2$ is the operator-function
\begin{multline}\notag
\sigma_Z(\mathcal{D})(z,\zeta) = \\ =
\left(\rule{0ex}{6ex}
    \begin{array}{ccc}
        \sigma_Z(D_0 + G_0 + G_1 + M_0) & \sigma_Z(C_1 + C_1^\prime) & \sigma_Z(C_2 + C_2^\prime) \vspace{1ex} \\
        \sigma_Z(B_1 + B_1^\prime) & \sigma_Z(D_1 + M_1) & \sigma_Z(T_{12}) \vspace{1ex} \\
        \sigma_Z(B_2 + B_2^\prime) & \sigma_Z(T_{21}) & \sigma_Z(D_2 + M_2) \\
    \end{array}
\right)(z,\zeta),
\end{multline}
where $(z,\zeta) \in T_0^*Z$, with values in operators acting in the spaces
\begin{equation}\notag
\sigma_Z(\mathcal{D})(z,\zeta)\colon
\begin{array}{c}
    H(\mathbb{R}^{\nu_3}) \\
    \oplus \\
    H(\mathbb{R}^{n_1}) \\
    \oplus \\
    H(\mathbb{R}^{n_2})
\end{array}
\longrightarrow
\begin{array}{c}
    H(\mathbb{R}^{\nu_3}) \\
    \oplus \\
    H(\mathbb{R}^{n_1}) \\
    \oplus \\
    H(\mathbb{R}^{n_2})
\end{array}.
\end{equation}
\end{enumerate}

\subsection{The composition formula}

In this section we show that  symbols of   morphisms are  well defined  and establish the corresponding
composition formula. This is the main result of the present paper.
\begin{theorem}\label{th:main}
For any morphism $\mathcal{D} \in \operatorname{Mor}$, its symbol $\sigma_Z(\mathcal{D})$ (see Definition~\ref{def:morphismSmb})
on any stratum $Z$ does not depend on the choice of representation of $\mathcal{D}$
in terms of generators~\eqref{eq:pdo},~\eqref{eq:b}, and~\eqref{eq:cob} of the algebra $\operatorname{Mor}$.
Moreover, for any two morphisms $\mathcal{D}_1, \mathcal{D}_1 \in \operatorname{Mor}$ the following composition formula holds
\begin{equation}\label{eq:comp1}
\sigma_Z(\mathcal{D}_1 \mathcal{D}_2) = \sigma_Z(\mathcal{D}_1) \, \sigma_Z(\mathcal{D}_2).
\end{equation}
\end{theorem}
\begin{proof}
Note that~\eqref{eq:comp1} readily follows from Definition~\ref{def:morphismSmb}.
Thus, we only have to prove that the symbol is well-defined
(that is, it does not depend on the choice of representation of a morphism
in terms of generators).

1. Define the following order reduction operators
\begin{equation}\notag
(\Lambda_0,\Lambda_1,\Lambda_2)\colon
\bigoplus_{k=0,1,2} H^{s_k}(X_k)\longrightarrow \bigoplus_{k=0,1,2} L^2(X_k),
\end{equation}
where $\Lambda_k$ is an elliptic $\psi$DO on $X_k$ of order $s_k$.
Now, we can reduce any morphism $\mathcal{D} \in \operatorname{Mor}$ to an operator acting in $L^2$-spaces
by multiplying it by appropriate powers of order reduction operators.
Then it suffices to prove the   theorem for the resulting operators;
the general case easily follows.

Thus, we assume that $\mathcal{D}\in\operatorname{Mor}$ acts in the spaces
\begin{equation}\label{eq:mor0}
\mathcal{D}\colon \bigoplus_{k=0,1,2} L^2(X_k)\longrightarrow \bigoplus_{k=0,1,2} L^2(X_k).
\end{equation}

2. Now, we show that symbols of the morphism~\eqref{eq:mor0} are well defined.
We note that our approach is based on the ideas described in~\cite{Hor7}
(for a smooth manifold without boundary)
and in~\cite{ReSc1}
(for boundary value problems).

First, we introduce an auxiliary operator family.
Namely, consider the space $\mathbb{R}^{k+\nu}$ with coordinates $(z,y)$.
Given a point $(z_0,\zeta_0) \in T_0^*\mathbb{R}^k$, define the operator family (cf.~\cite{ReSc1})
\begin{equation}\label{eq:RSchOperatorFamily}
R_{\lambda,z,y}\colon
L^2(\mathbb{R}^{k+\nu}_{z,y})
\longrightarrow
L^2(\mathbb{R}^{k+\nu}_{z,y}),\quad \lambda>0,
\end{equation}
where
\begin{equation}\notag
R_{\lambda,z,y}\colon
u(z,y)
\longmapsto
\lambda^{k/4+\nu/2} \, e^{i\lambda z \zeta_0} \, u\bigl(\lambda^{1/2}(z-z_0),\lambda y\bigr).
\end{equation}
A straightforward computation shows that the operators~\eqref{eq:RSchOperatorFamily}
are unitary, and for any $u\in C^\infty_c(\mathbb{R}^{k+\nu})$ the sequence of functions
$R_{\lambda,z,y}u$ tends to $0$ weakly in $L^2(\mathbb{R}^{k+\nu}_{z,y})$ as $\lambda\to\infty$.

Let us return to the morphism~\eqref{eq:mor0}. Consider its component
$$
\mathcal{D}_{kl}\colon L^2 (X_l) \longrightarrow L^2(X_k) 
$$
and its symbol $\sigma_Z(\mathcal{D}_{kl})$ on some stratum  $Z \subset X_1 \cup X_2$.
Choose the coordinates $(z,y) \in \mathbb{R}^{n+\nu} = X_k$, $(z,y') \in \mathbb{R}^{n+\nu'} = X_l$,
such that $Z$ is defined by the equations $Z = \{(z,0)\}$ in $X_k$ and in $X_l$.
Here $\nu$ stands for the codimension of $Z$ in $X_k$,
and $\nu'$ is the codimension of $Z$ in $X_l$.

The following lemma implies that the symbol $\sigma_Z(D_{kl})$ is well defined.
\begin{lemma}\label{lem:testing}
For any two functions
$u \in C^\infty_c(\mathbb{R}^n_z)$, $v \in C^\infty_c(\mathbb{R}^{\nu'}_{y'})$,
and any point $(z_0,\zeta_0) \in T^*_0Z$ the following equality holds
\begin{equation}\label{eq:ReSchDifference}
\lim_{\lambda\to\infty }\bigl\|R_{\lambda,z,y}^{-1} \, \mathcal{D}_{kl} \, R_{\lambda,z,y'} (u \otimes v)
- u \otimes \bigl[\sigma_{Z}(\mathcal{D}_{kl})(z_0,\zeta_0)\bigr]v\bigr\|_{L^2(X_k)}=0.
\end{equation}
\end{lemma}
\begin{proof}
By linearity and multiplicativity of the expression under the norm sign in~\eqref{eq:ReSchDifference}, it suffices to prove that the limit is equal to zero in the following
three special cases:
\begin{enumerate}
\item $k=l$, and $\mathcal{D}_{kk}$ is a $\psi$DO of order zero;
\item $l=0$, $k>0$, and
\begin{equation}\notag
\mathcal{D}_{k0} = \Lambda_k \, i^k \, \Lambda_0
\colon L^2(X_0)\longrightarrow L^2(X_k),
\end{equation}
where $\Lambda_0$ and  $\Lambda_k$ are order reduction operators on $X_0$ and $X_k$ respectively.
\item  $k=0$, $l>0$, and
\begin{equation}\notag
\mathcal{D}_{0l}= \Lambda_0 \, i_l \, \Lambda_l
\colon L^2(X_l)\longrightarrow L^2(X_0),
\end{equation}
where $\Lambda_0$ and  $\Lambda_l$ are order reduction operators on $X_0$ and  $X_l$  respectively.
\end{enumerate}

For brevity, we omit the details of the corresponding verification
and refer the reader to~\cite{ReSc1} where a  similar calculation is carried out.
\end{proof}

Now, it follows from Lemma~\ref{lem:testing} that the symbol of~\eqref{eq:mor0}
is well defined. In turn, this implies that  symbols of   general morphisms are well defined.

The proof of Theorem~\ref{th:main} is now complete.
\end{proof}


\begin{thebibliography}{10}

\bibitem{Ster2}
B. Sternin.
\newblock Elliptic (co)boundary morphisms.
\newblock {\em Soviet Math. Dokl.}, 8(1):41--45, 1967.

\bibitem{Ster5}
B. Sternin.
\newblock Elliptic morphisms (riggings of elliptic operators) for submanifolds with singularities.
\newblock {\em Soviet Math. Dokl.}, 12:1338--1343, 1971.

\bibitem{ZeSt1}
M.~I.~Zelikin, B.~Yu.~Sternin,
\newblock A system of integral equations that arises in the problem of S.L. Sobolev.
\newblock {\em Siberian Math. J.} 18(1):72--76, 1977.

\bibitem{SaSt32}
A.~Yu.~Savin, B.~Yu.~Sternin.
\newblock Elliptic translators on manifolds with multidimensional singularities.
\newblock {\em Differ. Equations}, 49(4):494--509, 2013.

\bibitem{SaSt35}
A.~Yu.~Savin, B.~Yu.~Sternin.
\newblock Index of Sobolev problems on manifolds with manydimensional singularities.
\newblock {\em Differ. Equations}, 50(2):232--245, 2014.

\bibitem{Linh1}
L.~L.~Nguyen.
\newblock On Fredholm riggings of $G$-translators.
\newblock {\em Differ. Equations}, 48(8):1192--1196, 2012.

\bibitem{Losh1}
D.~A.~Loshchenova,
\newblock Sobolev problems associated with Lie group actions.
\newblock {\em Differ. Equations}, 51(8):1051-1064, 2015.

\bibitem{Ster1}
B.~Sternin.
\newblock Elliptic and parabolic problems on manifolds with boundary consisting
  of components of different dimension.
\newblock {\em Trans. Moscow Math. Soc.}, 15:387--429, 1966.

\bibitem{StSh53-1}
B.~Sternin and V.~Shatalov.
\newblock Relative elliptic theory and the {S}obolev problem.
\newblock Sbornik: Mathematics,   187(11):1691--1720, 1996.

\bibitem{NaSt10}
V. Nazaikinskii and B. Sternin.
\newblock Relative elliptic theory.
\newblock In J. Gil, Th. Krainer, and I. Witt, editors, {\em Aspects of
  {B}oundary {P}roblems in {A}nalysis and {G}eometry}, volume 151 of {\em
  Operator {T}heory: {A}dvances and {A}pplications. {A}dvances in {P}artial
  {D}ifferential {E}quations}, pages 495--560, Basel--Boston--Berlin, 2004.
  Birkh{\"a}user.


\bibitem{NaSt13}
V.~E.~Nazaikinskii, B.~Yu.~Sternin.
\newblock On the {G}reen operator in relative elliptic theory.
\newblock {\em Dokl. Math.}, 68(1):57--60, 2003.

\bibitem{NoSt1}
S.~Novikov, B.~Sternin.
\newblock Traces of elliptic operators on submanifolds and {$K$-theory}.
\newblock {\em Soviet Math. Dokl.}, 7(5):1373--1376, 1966.

\bibitem{Hor7}
L.~H{\"o}rmander.
\newblock Pseudo-differential operators.
\newblock {\em Comm. Pure Appl. Math.}, 18:501--517, 1965.

\bibitem{ReSc1}
S.~Rempel, B.-W. Schulze.
\newblock {\em Index Theory of Elliptic Boundary Problems}.
\newblock Akademie--Verlag, Berlin, 1982.


\end{thebibliography}
  

\end{document}